\documentclass[10pt, reqno]{amsart}
\usepackage{amsfonts}
\usepackage{amsmath,amssymb,appendix}
\usepackage{mathrsfs}
\usepackage{hyperref}
\usepackage{graphicx}
\usepackage{amsthm}
\usepackage{color}
\usepackage{dsfont}
\usepackage{bbm, dsfont}
\usepackage{cite}
\numberwithin{equation}{section}

\usepackage{hyperref}

\allowdisplaybreaks

\newcommand{\R}{\mathbb{R}}

\newtheorem{theorem}{Theorem}[section]

\newtheorem*{theorem*}{Theorem}

\newtheorem{lemma}[theorem]{Lemma}

\newtheorem{corollary}[theorem]{Corollary}

\newtheorem{proposition}[theorem]{Proposition}

\theoremstyle{definition}
\newtheorem{definition}[theorem]{Definition}

\theoremstyle{remark}

\begin{document}

\title[On dispersive decay for gKdV]{On dispersive decay for the generalized Korteweg--de Vries equation}

\author[M.~Kowalski]{Matthew Kowalski}
\address{Department of Mathematics, University of California, Los Angeles, CA 90095, USA}
\email{mattkowalski@math.ucla.edu}

\author[M.~Shan]{Minjie Shan}
\address{College of Science, Minzu University of China, Beijing 100081, P. R. China}
\email{smj@muc.edu.cn}

\begin{abstract}
    We prove pointwise-in-time dispersive estimates for solutions to the generalized Korteweg--de Vries (gKdV) equation.
    In particular, for solutions to the mass-critical model, we assume only that initial data lie in $\dot{H}^{\frac{1}{4}} \cap \dot{H}^{-\frac{1}{12}}$ and show that solutions decay in $L^\infty$ like $|t|^{-\frac{1}{3}}$. To accomplish this, we develop a persistence of negative regularity for solutions to gKdV and extend Lorentz--Strichartz estimates to the mixed norm case.
\end{abstract}

\maketitle

\section{Introduction}
    \noindent We study dispersive decay for the {\em generalized Korteweg--de Vries equation}:
\begin{equation}\label{gKdV}\tag{gKdV}
	\begin{cases}
		\partial_{t}u +\partial_{x}^{3} u \pm \partial_{x} u^{k+1} = 0, \\
		u(0,x)=u_0(x),  \\
    \end{cases}
\end{equation}
where $u : \R_t \times \R_x \to \R$ is a real-valued function on spacetime and $k \geq 4$ is an integer. With this notation, $+$ corresponds to the focusing equation and $-$ to the defocusing equation. Beyond the global theory, our methods are blind to the focusing nature of \eqref{gKdV}.

Like many dispersive models, \eqref{gKdV} enjoys a scaling symmetry ---
$ u(t, x) \mapsto u_{\lambda}(t, x) = \lambda^{\frac{2}{k}} u ( \lambda^{3}t, \lambda x) $, for $\lambda > 0$ --- that preserves the class of solutions. This scaling symmetry defines a critical regularity $s_k = \frac{1}{2} - \frac{2}{k}$. Specifically, $\dot{H}^{s_k}$ is a dimensionless measure of size for \eqref{gKdV}. Notably, for $k = 4$, $s_k = 0$ and so \eqref{gKdV} is {\em mass-critical}: the conserved mass
\begin{equation*}
    M(u)=\int_{\R}|u(t, x)|^2 dx = \int_\R |u_0|^2 dx
\end{equation*}
is preserved under scaling. As the critical regularity aligns with a conserved quantity, strengthened global results are available when $k = 4$, which we leverage later.

In their influential paper \cite{KPV93}, Kenig, Ponce, and Vega proved that \eqref{gKdV} is globally well-posed and scatters in $\dot{H}^{s_k}(\R)$ for small initial data. In the following theorem, we recall the necessary well-posedness results for \eqref{gKdV} from \cite{KPV93}:
\begin{theorem}[Well-posedness and scattering] \label{KPV-gKdV-Global-kg4}
Let $k\geq4$. Then there exists $\delta_k>0$ such that for any $u_0 \in \dot{H}^{s_k}(\R)$ with
$\|u_0\|_{\dot{H}^{s_k}}<\delta_k$,
there exists a unique global solution $u(t)$ of \eqref{gKdV} with initial datum $u_0$ which satisfies:
	\begin{align}
	u\in C\big(\R; \dot{H}^{s_k}(\R) \big)& \cap L^{\infty}\big(\R; \dot{H}^{s_k}(\R) \big), \nonumber \\
	\left\||\partial_x|^{s_k} u_x \right\|_{L^{\infty}_{x}L^{2}_{t}}\leq C\big(\|u_0\|_{\dot{H}^{s_k}}\big),
\;\;&\;\;	\||\partial_x|^{s_k}u\|_{L^{5}_{x}L^{10}_{t}}\leq C\big(\|u_0\|_{\dot{H}^{s_k}}\big),\label{KPV-gKdV-Global-kg4-A2} \\
		\big\||\partial_x|^{\frac{1}{10}-\frac{2}{5k}}|\partial_t|^{\frac{3}{10}-\frac{6}{5k}}u&\big\|_{L^{p_k}_{x}L^{q_k}_{t}} \leq C\big(\|u_0\|_{\dot{H}^{s_k}}\big), \label{KPV-gKdV-Global-kg4-A3}
	\end{align}
    where $\frac{1}{p_k}=\frac{2}{5k}+\frac{1}{10}$, $\frac{1}{q_k}=\frac{3}{10}-\frac{4}{5k}$.

    Moreover, there exist $u_0^{\pm}\in \dot{H}^{s_k}(\R)$ such that
	\begin{align}\label{KPV-gKdV-GlobalscatteringA4}
		\lim_{t\to \pm \infty}\left\|u(t)- e^{-t\partial^3_x}u_0^{\pm}\right\|_{\dot{H}^{s_k}_{x}}=0.
	\end{align}
    In this way, we say that the global strong solution $u(t)$ scatters in $\dot{H}^{s_k}(\R)$ to a solution of the linear equation as $t\to \pm \infty$.
\end{theorem}
We note that the scattering result \eqref{KPV-gKdV-GlobalscatteringA4} for $s_k \neq 0$ is not stated in the main theorems of \cite{KPV93}, but rather, is remarked on  in \cite[\S1]{KPV93}.

For the defocusing mass-critical equation, Dodson extended global well-posedness and scattering to arbitrary initial data $u_0 \in L^2(\R)$ with the global spacetime bound \cite{Dodson2017}:
\begin{align}\label{gKdV-Dodson}
    \| u \|_{L^{5}_{x}L^{10}_{t}(\R\times\R)}\leq C\big(M(u_0)\big). 
\end{align}
For future analysis, we note that this aligns with estimate \eqref{KPV-gKdV-Global-kg4-A3}.

Scattering \eqref{KPV-gKdV-GlobalscatteringA4} indicates that global solutions of \eqref{gKdV} asymptotically parallel those of the Airy evolution.
It is therefore natural to compare the long-time behavior of solutions to \eqref{gKdV} with those of the Airy equation; indeed, both satisfy the spacetime bounds \eqref{KPV-gKdV-Global-kg4-A2}.
In particular, the Airy evolution is well known to exhibit a pointwise quantitative decay,
\begin{align}\label{jcKdV-DisDecay}
    \big\|e^{-t\partial^3_x}f\big\|_{L^\infty} \lesssim  |t|^{-\frac{1}{3}} \|f\|_{L^1},
\end{align}
which raises the question of whether this decay persists for solutions to \eqref{gKdV} in spite of the nonlinearity, through scattering.

In our main theorem, we answer this question and show an analog of \eqref{jcKdV-DisDecay} for solutions to \eqref{gKdV}:
\begin{theorem}\label{theorem}
    Fix $4 \leq k < 8$. Given $u_0 \in L^{1}\cap\dot{H}^{\frac{1}{4}}(\R)\cap \dot{H}^{\frac{1}{4} - \frac{8-k}{4(k-1)}}$ which satisfies the hypotheses for global existence for \eqref{gKdV}, let $u(t)$ denote the unique global solution to \eqref{gKdV} with initial datum $u_0$. Then
    \begin{equation}\label{theorem/decay-1}
        \|u(t)\|_{L_x^\infty} \leq C\big(\|u_0\|_{\dot{H}^{\frac{1}{4}}}, \|u_0\|_{\dot{H}^{\frac{1}{4}-\frac{8-k}{4(k-1)}}}\big) |t|^{-\frac{1}{3}} \|u_0\|_{L^{1}},
    \end{equation}
    uniformly for $t \neq 0$.
    In addition, for $k \geq 8$, we find
    \begin{equation}\label{theorem/decay-2}
        \|u(t)\|_{L_x^\infty} \leq C\big(\|u_0\|_{\dot{H}^{s_k}}\big) |t|^{-\frac{1}{3}} \|u_0\|_{L^{1}},
    \end{equation}
    uniformly for $t \neq 0$.
\end{theorem}
This result is optimal for $k \geq 8$ as it assumes no requirements on the initial data beyond belonging to the scaling-critical $\dot{H}^{s_k}$ and those required for global existence. For the case $4 \leq k < 8$, we must make additional assumptions on the regularity of the initial data. In particular, for the mass-critical \eqref{gKdV} ($k = 4$, $s_k = 0$), we find that
\begin{equation*}\label{theorem/decay}
    \|u(t)\|_{L_x^\infty} \leq C\big(\|u_0\|_{\dot{H}^{\frac{1}{4}}}, \|u_0\|_{\dot{H}^{-\frac{1}{12}}}\big) |t|^{-\frac{1}{3}} \|u_0\|_{L^{1}},
\end{equation*}
uniformly for $t \neq 0$ and for all $u_0 \in \dot{H}^{\frac{1}{4}}\cap \dot{H}^{-\frac{1}{12}}$. The condition that $u_0 \in \dot{H}^{-\frac{1}{12}}$ requires that we develop a persistence of negative regularity for \eqref{gKdV}; see Proposition \ref{spacetime/theorem}.

For $4 \leq k < 8$, we note that $L^1 \cap \dot{H}^{\frac{1}{4}} \hookrightarrow \dot{H}^{\frac{1}{4} - \frac{8-k}{4(k-1)}}$ and so Theorem \ref{theorem} can be restated without the dependence on $\dot{H}^{\frac{1}{4} - \frac{8-k}{4(k-1)}}$. In doing so, the linear dependence on $\|u_0\|_{L^1}$ is lost.
\begin{corollary}
    Fix $4 \leq k < 8$. Given $u_0 \in L^{1}\cap\dot{H}^{\frac{1}{4}}(\R)$ which satisfies the hypotheses for global existence for \eqref{gKdV}, let $u(t)$ denote the unique global solution to \eqref{gKdV} with initial datum $u_0$. Then
    \begin{equation*}
        \|u(t)\|_{L_x^\infty} \leq C\big(\|u_0\|_{L^1}, \|u_0\|_{\dot{H}^{\frac{1}{4}}}\big) |t|^{-\frac{1}{3}},
    \end{equation*}
    uniformly for $t \neq 0$.
\end{corollary}

Estimates of this form have received substantial attention in the study of dispersive equations. Indeed, before the introduction of Strichartz estimates, quantitative estimates pointwise-in-time were the primary method for understanding long-time behavior. Consequently, they were originally achieved only under strong regularity and decay assumptions; see, e.g., \cite{LinStr78, Klai85, KlaiPon83, MorStr72} and references therein.

In recent years, there has been a return to these estimates with modern scaling-critical global well-posedness results, such as those obtained by Dodson \cite{Dodson2017}. 
In \cite{FKVZ24}, Fan, Killip, Vi\c{s}an and Zhao proved dispersive decay for global solutions to the {\em mass-critical nonlinear Schr\"odinger equation}, assuming only that initial data belonged to the scaling-critical $L^2$.
We believe their work is optimal: no assumptions are made on the initial besides belonging to the critical space, a linear dependence on the initial data is recovered, and the constants depend only on the size of the initial data.

These results have since been expanded: In \cite{Kowal24, Kowal25}, the first author obtained pointwise-in-time dispersive decay for solutions to the {\em energy-critical nonlinear Schr\"odinger equation} and the {\em energy-critical nonlinear wave equation} with initial data in $\dot{H}^1$ and the Besov space $\dot{B}^1_{2,1}$. In \cite{Shan24}, the second author proved dispersive decay for solutions to \eqref{gKdV} for suitably small initial data in $H^{\frac{1}{2}}$, as well as for solutions to the {\em generalized Zakharov--Kuznetsov} under suitable conditions on the initial data. In this work, we build on \cite{Shan24} and sharpen the requirements for dispersive decay for \eqref{gKdV}.
    \subsection*{Notation}
        We use the usual notation $A \lesssim B$ to indicate that $A \leq C B$ for some universal constant $C > 0$ that will change from line to line. If both $A \lesssim B$ and $B \lesssim A$ then we use the notation $A \sim B$. When the implied constant fails to be universal, the relevant dependencies will be indicated within the text or included as subscripts on the symbol. 

For notation purposes, we abbreviate the maximum and minimum of two numbers $a$ and $b$ as $a \vee b$ and $a \wedge b$ respectively.

We define the Fourier transform as
\begin{equation*}
\widehat{f}(\xi) = \tfrac{1}{\sqrt{2\pi}}\int e^{-i \xi x} f(x) dx \quad \text{so} \quad f(x) = \tfrac{1}{\sqrt{2\pi}} \int e^{i \xi x} \widehat{f}(\xi) d\xi.
\end{equation*}
Concomitant with this, for $s \in \R$, we define the homogeneous Sobolev space $\dot{H}^s$ as the completion of the Schwartz functions $\mathcal{S}(\R)$ with respect to the norm
\begin{equation*}
    \|f\|^2_{\dot{H}^s(\R)} = \int_\R |\xi|^{2s} |\widehat{f}(\xi)|^2 d\xi.
\end{equation*}
When $s < -\frac{1}{2}$, we restrict to Schwartz functions that vanish at $\xi = 0$. 

We use $L_t^p L_x^q(T \times X)$ to denote the {\em forward} mixed Lebesgue spacetime norm
\begin{equation*}
    \|f\|_{L_t^p L_x^q(T\times X)} = \big\| \|f(t,x)\|_{L^q(X,dx)} \big\|_{L^p(T,dt)} = \bigg[ \int_T \bigg(\int_X |f(t,x)|^q dx\bigg)^{\frac{p}{q}} dt\bigg]^{\frac{1}{p}}.
\end{equation*}
Similarly, we define the {\em backward} mixed Lebesgue spacetime norm $L_x^q L_t^p$ in the obvious way. When $X = \R$, we let $L_t^pL_x^q(T) = L_t^pL_x^q(T \times \R)$,
and when $T = X = \R$, we let $L_t^pL_x^q = L_t^pL_x^q(T \times \R)$, unless otherwise stated. Similarly for backwards norms, we define $L_x^qL_t^p(T) = L_x^qL_t^p(\R \times T)$. For our purposes, the spatial norm will always be taken over $X = \R$. Finally, when $p = q$, we let $L^p_{t,x} = L_t^p L_x^p = L^p_x L^p_t$.

\subsection*{Acknowledgments}
M.K.~is supported, in part, by NSF grants DMS-2154022, DMS-2054194, and DMS-2348018. M.S.~is partially supported by the National Natural Science Foundation of China, grant numbers: 12101629 and 12371123.

\section{Preliminaries}\label{section:linear estimate}
    In this section, we recall the harmonic analysis results and the Airy evolution estimates which will be necessary for our proof of Theorem \ref{theorem}.
\subsection{Lorentz spaces}
We define Lorentz spaces in the following way. For a textbook treatment of these spaces, we direct the interested reader to \cite{Grafa14}.
    
    \begin{definition}[Lorentz spaces] Fix $1\leq p<\infty$ and $1\leq q\leq\infty$. The Lorentz space $L^{p,q}(\R^d)$ is the space of measurable functions $f:\R^d \to \mathbb{C}$ which have finite quasinorm:
    \begin{equation}\label{prel/quasinorm}
        \|f\|_{L^{p,q}(\R^d)}=p^{\frac{1}{q}}\left\| \lambda \big| \{x\in \R^d: |f(x)|>\lambda \}\big|^{\frac{1}{p}} \right\|_{L^{q}((0, \infty), \frac{d\lambda}{\lambda})},
    \end{equation}
    where $|*|$ denotes the Lebesgue measure on $\R^d$.
    \end{definition}
    
    We note that $L^{p,p}(\R^d)$ coincides with the usual Lebesgue space $L^{p}(\R^d)$, and that $L^{p,\infty}(\R^d)$ coincides with the weak-$L^{p}(\R^d)$ space. In addition, the Lorentz spaces satisfy a natural nesting: if $1\leq q<r\leq\infty$, then $L^{p,q}(\R^d) \hookrightarrow  L^{p,r}(\R^d)$.

    Central to our analysis is the observation:
    \begin{equation*}
        \big\||t|^{-\frac{1}{p}}\big\|_{L_t^{p,\infty}(\R)} \sim 1,
    \end{equation*}
    uniformly in $p$. Indeed, this observation is the reason we introduce Lorentz norms in our analysis and is the full extent to which we will use the exact form of \eqref{prel/quasinorm}.
    
    Lorentz spaces enjoy many of the same estimates as Lebesgue spaces. In particular, we find the following forms of H\"older's inequality and Young's inequality, here called the Young--O'Neil convolutional inequality \cite{young-oneil-proof,lorentz-strichartz,young-oneil-introduction}:
    \begin{lemma}[H\"older's inequality] Let $1\leq p,p_1,p_2<\infty$, $1\leq q,q_1,q_2\leq\infty$, and $\frac{1}{p}=\frac{1}{p_1}+\frac{1}{p_2}$, $\frac{1}{q}=\frac{1}{q_1}+\frac{1}{q_2}$. Then,
    $$\|fg\|_{L^{p,q}(\R^d)}\lesssim \|f\|_{L^{p_1,q_1}(\R^d)}\|g\|_{L^{p_2,q_2}(\R^d)}.$$
    \end{lemma}
    
    \begin{lemma}[Young--O'Neil convolutional inequality]
        Given $1 < p,p_1,p_2 < \infty$ and $1 \leq q, q_1,q_2 \leq \infty$ such that
        $\frac{1}{p} + 1 = \frac{1}{p_1} + \frac{1}{p_2}$ and $\frac{1}{q} = \frac{1}{q_1} + \frac{1}{q_2}$,
        \begin{equation*}
            \|f*g\|_{L^{p,q}} \lesssim_{d,p_i,q_i} \|f\|_{L^{p_1,q_1}} \|g\|_{L^{p_2,q_2}}.
        \end{equation*}
    \end{lemma}

\subsection{Backwards mixed norms}
    Due to the derivative in the nonlinearity of \eqref{gKdV}, a central tool in our analysis will be the use of backwards mixed norms. We recall the necessary theory here.
    
    These spaces allow for the use of the Kato smoothing effect and the maximal estimate, which we recall in the following Lemma:
    \begin{lemma}[Kato smoothing and maximal function estimate, \cite{KPV93}] 
        The Airy evolution satisfies: 
        \begin{align}
            \big\|\partial_x e^{-t\partial^3_x}f \big\|_{L^\infty_xL^2_{t}} & \lesssim \|f\|_{L^{2}_{x}}, \label{prel/kato}\\
            \big\|e^{-t\partial^3_x} f\big\|_{L_{x}^{4} L_{t}^{\infty}} & \lesssim  \|u_0\|_{\dot{H}^{\frac{1}{4}}_x}. \label{prel/maximal}
    	\end{align}
    \end{lemma}
    
    These estimates combine to give a range of Strichartz estimates in the backwards mixed norms $L_x^pL_t^q$:
    \begin{lemma} \label{gKdVL(XT)Norm}
    Let $2 \leq p,q \leq \infty$ satisfy $\frac{4}{p}+\frac{2}{q}=1$. Then
    	\begin{equation}
    		\big\||\partial_x|^{1-\frac{5\theta}{4}}e^{-t\partial^3_x} u_0\big\|_{L_{x}^{p} L_{t}^{q}} \lesssim  \|u_0\|_{L^{2}_x}.
    	\end{equation}
    Moreover, for $2\leq p_i,q_i  \leq \infty$ satisfying $\frac{4}{p_1}+\frac{2}{q_1}=\frac{4}{p_2}+\frac{2}{q_2}=1$, we may estimate
    	\begin{equation}
    		\left\||\partial_x|^{2-\frac{5}{p_1}-\frac{5}{p_2}}\int_{0}^{t} e^{-(t-s)\partial^3_x}g(s) ds\right\|_{L_{x}^{p_1} L_{t}^{q_1}} \lesssim  \|g\|_{L_{x}^{p_2'} L_{t}^{q_2'}}.
    	\end{equation}
    \end{lemma}

    To aid in the proof of Proposition \ref{spacetime/theorem},  we also recall the fractional Leibniz rule in backward mixed norms; see \cite[Theorem A.8]{KPV93}. For our purposes, the range of exponents needed can also be obtained by adapting the standard Littlewood--Paley proof to Banach space valued functions; we direct the interested reader to \cite{operator-valued} and ~\cite[$\S$1.6.4]{Stein93} for the necessary theory.
    \begin{lemma}[Fractional Leibniz rule]\label{proof/leibniz}
        Let $s \geq 0$ and $1 < p,p_i,q_i < \infty$ for $i \in \{1,2,3,4\}$ such that $\frac{1}{p} = \frac{1}{p_1} + \frac{1}{p_2} = \frac{1}{p_3} + \frac{1}{p_4}$ and $\frac{1}{q} = \frac{1}{q_1} + \frac{1}{q_2} = \frac{1}{q_3} + \frac{1}{q_4}$. Then
        \begin{equation*}
            \big\| |\partial_x|^s fg \big\|_{L_x^{q} L_t^{p}} \lesssim \big\| |\partial_x|^sf\big\|_{L_x^{q_1} L_t^{p_1}}\big\| g \big\|_{L_x^{q_2} L_t^{p_2}} + \big\| f \big\|_{L_x^{q_3} L_t^{p_3}}\big\| |\partial_x|^sg \big\|_{L_x^{q_4} L_t^{p_4}}.
        \end{equation*}
    \end{lemma}
    
    We also recall a Sobolev embedding in backward mixed spaces. This can be shown by adapting the usual arguments to Banach space valued functions.
    \begin{lemma}[Sobolev embedding]\label{proof/sobolev}
        Fix $1 < p,q,r < \infty$ and $s \geq 0$ such that $\frac{1}{q} + \frac{s}{d} = \frac{1}{r}$. Then
        \begin{equation*}
            \|f\|_{L_x^q L_t^p} \lesssim \big\||\partial_x|^s f\big\|_{L_x^r L_t^p}.
        \end{equation*}
    \end{lemma}

\subsection{Endpoint Leibniz rule}
    Finally, it will be necessary to employ endpoint fractional Leibniz rules. For our purposes, the work of Dong Li \cite{endpoint-Leibniz} is sufficient and the necessary results are stated in the following proposition. For further endpoint results, we direct the interested reader to \cite{endpoint-Leibniz-4,endpoint-Leibniz-3, endpoint-Leibniz-2} and references therein.
    \begin{proposition}[Endpoint Leibniz rule]\label{prel/leibniz}
        Let $s > 0$ and $1 < p,q < \infty$. Then
        \begin{equation*}
            \||\partial_x|^s fg\|_{L^1} \lesssim_{p,s} \big\||\partial_x|^sf\big\|_{L^p}\|g\|_{L^{p'}} + \|f\|_{L^q}\big\||\partial_x|^s g\big\|_{L^{q'}}.
        \end{equation*}
        Additionally, for $1 < p < \infty$,
        \begin{equation*}
            \||\partial_x|^s fg\|_{L^p} \lesssim_{p,s} \big\||\partial_x|^sf\big\|_{L^{q_1}}\|g\|_{L^{r_1}} + \|f\|_{L^{q_2}}\big\||\partial_x|^s g\big\|_{L^{r_2}},
        \end{equation*}
        where $1 < q_i, r_i \leq \infty$ and $\frac{1}{q_i} + \frac{1}{r_i} = \frac{1}{p}$. Notably, this includes $q_1 = \infty$ or $r_2 = \infty$.
    \end{proposition}
    For instance, in the analysis of the mass-critical \eqref{gKdV}, we will apply the preceding estimates twice to find
    \begin{equation}\label{prel/endpoint product}
        \big\| |\partial_x|^\frac{1}{2} u^5 \big\|_{L^1} \lesssim \|u\|_{L^\infty} \|u\|_{L^3}^3 \big\||\partial_x|^\frac{1}{2} u\big\|_{L^\infty};
    \end{equation}
    see the progression from \eqref{proof/failure 0} to \eqref{proof/failure 1}.

\section{Lorentz--Strichartz estimates}
    In this section, we prove Lorentz--Strichartz estimates for solutions to \eqref{gKdV}. First observed in \cite{lorentz-strichartz}, Fan, Killip, Vi\c{s}an, and Zhao later used these estimates to control the time decay in the study of dispersive estimates \cite{FKVZ24}. Such estimates were refined in \cite{Kowal24} to include Lorentz exponents below $1$ and Lorentz spatial norms.

In this section, we extend these Lorentz--Strichartz estimates further, allowing the left- and right-hand spacetime norms to differ. This is achieved only for suitable spacetime norms, dictated by the interpolation argument used. At the moment, we know of no extensions to the range of admissible spaces on the left-hand side, though we expect such estimates to hold in general.

Fundamentally, Strichartz estimates are an extension of the linear dispersive estimates to the nonlinear \eqref{gKdV}. We then begin by recalling the standard dispersive estimates for the Airy evolution $e^{-t\partial^3_x}$ \cite{KPV89}:
\begin{lemma}[Dispersive estimates]
For $\theta \in [0,1]$, $r = \frac{2}{1-\theta}$ and $0\leq \alpha \leq \frac{1}{2}$,
    \begin{align}
        \left\||\partial_x|^{\theta \alpha}e^{-t\partial^3_x}u_0\right\|_{L^r_x} \lesssim  |t|^{-\frac{\alpha+1}{3}\theta } \left\|u_0\right\|_{L^{r'}_x}, \label{prel/dispersive}
    \end{align}
uniformly for $t \neq 0$. Moreover,
\begin{align}
        \left\|\partial_xe^{-t\partial^3_x}f\right\|_{L^{\infty}_x} \lesssim  |t|^{-\frac{1}{2}} \left\||\partial_x|^{\frac{1}{2}}f\right\|_{L^{1}_x}. \label{prel/dispersive-adapted}
    \end{align}
\end{lemma}
We note that the estimate \eqref{prel/dispersive-adapted} is nonstandard, but follows from a straightforward adaptation of the proof of \eqref{prel/dispersive}; see \cite{erdogan-tzirakis} for a textbook treatment. 
From the dispersive estimates, along with the Young--O'Neil convolutional inequality, we find the following standard Lorentz--Strichartz estimates:
\begin{lemma}\label{prel/strichartz-basic}
    Let $\theta \in [0,1)$, $\alpha \in [0,\frac{1}{2}]$, $(q,p) = (\frac{6}{\theta(\alpha + 1)},\frac{2}{1-\theta})$ and $\frac{1}{p} + \frac{1}{p'} = \frac{1}{q} + \frac{1}{q} = 1$. Then for any spacetime slab $J \times \R$ and $t_0 \in J$,
    \begin{equation*}
        \bigg\|\int_{t_0}^t |\partial_x|^\frac{\theta\alpha}{2} e^{-(t-s)\partial_x^3} g(s,\cdot) ds\bigg\|_{L_t^{q,2}L_x^{p}(J)} \lesssim \big\||\partial_x|^{-\frac{\theta\alpha}{2}}g\big\|_{L_t^{q',2}L_x^{p'}(J)}.
    \end{equation*}
\end{lemma}

In the works \cite{Kowal24, FKVZ24}, these estimates were sufficient to control all needed spacetime norms. However, due to the derivative in the nonlinearity, for \eqref{gKdV} it is important to choose symmetric spacetime norms --- of the form $L^p_{t,x}$ --- to convert from forwards to backwards mixed spaces; see \eqref{proof/8}.
\begin{proposition}[Lorentz--Strichartz estimates]\label{prel/strichartz}
    Let $\theta_1,\theta_2 \in (0,1)$, $\alpha \in [0,\frac{1}{2}]$, $(q_i,p_i) = (\frac{6}{\theta_i(\alpha + 1)},\frac{2}{1-\theta_i})$ for $i \in \{1,2\}$, and $\frac{1}{p_2} + \frac{1}{p_2'} = \frac{1}{q_2} + \frac{1}{q_2'} = 1$.

    Suppose that $q_1 \leq q_2$ and $q_1 < \infty$. Then for any spacetime slab $J \times \R$ and $t_0 \in J$,
    \begin{equation*}
        \bigg\|\int_{t_0}^t |\partial_x|^\frac{\theta_1\alpha}{2} e^{-(t-s)\partial_x^3} g(s,\cdot) ds\bigg\|_{L_t^{q_1,2}L_x^{p_1}(J)} \lesssim \big\||\partial_x|^{-\frac{\theta_2\alpha}{2}}g\big\|_{L_t^{q_2'}L_x^{p_2'}(J)}.
    \end{equation*}

    Instead, suppose that $q_2 \leq q_1$. Then for any spacetime slab $J \times \R$ and $t_0 \in J$,
    \begin{equation*}
        \bigg\|\int_{t_0}^t |\partial_x|^\frac{\theta_1\alpha}{2} e^{-(t-s)\partial_x^3} g(s,\cdot) ds\bigg\|_{L_t^{q_1,\frac{2q_1}{q_2}}L_x^{p_1}(J)} \lesssim \big\||\partial_x|^{-\frac{\theta_2\alpha}{2}}g\big\|_{L_t^{q_2',2}L_x^{p_2'}(J)}.
    \end{equation*}
\end{proposition}
\begin{proof}
    We mimic the usual proof of non-endpoint Strichartz estimates. Throughout the proof, we estimate over $\R \times \R$; the restriction to the spacetime slab $J \times \R$ can be achieved by replacing $g(s,x)$ by $g(s,x) \mathds{1}_{s \in J}$.

    We first aim to show the $q_2 = \infty, p_2 = 2$ estimate. Arguing by duality, we find
    \begin{align*}
        \bigg\|\int_{t_0}^t |\partial_x|^\frac{\theta_1\alpha}{2} e^{-(t-s)\partial_x^3}&  g(s,\cdot) ds\bigg\|_{L_t^{q_1,2}L_x^{p_1}} \\
        & = \sup\bigg|\bigg\langle\int_{t_0}^t |\partial_x|^\frac{\theta_1\alpha}{2} e^{-(t-s)\partial_x^3} g(s,x) ds,  f(t,x)\bigg\rangle_{t,x}\bigg| \\
        & = \sup\bigg|\int \int_{t_0}^t\Big\langle  e^{s\partial_x^3} g(s,x),|\partial_x|^\frac{\theta_1\alpha}{2}e^{t\partial_x^3}f(t,x)\Big\rangle_{x}dsdt\bigg| \\
        & \leq \sup\big\|e^{s\partial_x^3} g(s,x)\big\|_{L_s^1L_x^2}\bigg\|\int |\partial_x|^\frac{\theta_1\alpha}{2}e^{t\partial_x^3}f(t,x)dt\bigg\|_{L_x^2}.
    \end{align*}
    where the supremum is taken over $\big\{f : \|f\|_{L_t^{q_1',2}L_x^{p_1'}} = 1\big\}$ and $\langle\cdot,\cdot\rangle_*$ is the $L^2$ inner product with subscript indicating the variable(s) over which the inner product is taken.

    By the dual estimate of Lemma \ref{prel/strichartz-basic}, we then find that
    \begin{align*}
        \bigg\|\int_{t_0}^t |\partial_x|^\frac{\theta_1\alpha}{2} e^{-(t-s)\partial_x^3} g(s,\cdot) ds\bigg\|_{L_t^{q_1,2}L_x^{p_1}}
        & \lesssim \sup\|g(s,x)\|_{L_s^1L_x^2}\|f\|_{L_t^{q_1',2}L_x^{p_1'}}\\
        & = \|g(s,x)\|_{L_s^1 L_x^2}.
    \end{align*}
    Interpolating with the estimates from Lemma \ref{prel/strichartz-basic}, this completes the proof of Proposition \ref{prel/strichartz} in the case $q_1 \leq q_2$.

    We now consider $q_1 = \infty$, $p_1 = 2$. Arguing again by duality, we may estimate
    \begin{align*}
        \bigg\|\int_{t_0}^t e^{-(t-s)\partial_x^3} g(s,\cdot) ds\bigg\|_{L_t^\infty L_x^2}
        & = \sup\bigg|\int \int_{t_0}^t\Big\langle  e^{s\partial_x^3} g(s,x),e^{t\partial_x^3}f(t,x)\Big\rangle_{x}dsdt\bigg| \\
        & \leq \sup\big\|e^{t\partial_x^3} f(t,x)\big\|_{L_t^1L_x^2}\bigg\|\int_{t_0}^t e^{s\partial_x^3}g(s,x)ds\bigg\|_{L_x^2} \\
        & \lesssim \Big\||\partial_x|^\frac{-\theta_2\alpha}{2} g(t,x) \Big\|_{L_t^{q_2',2}L_x^{p_2'}}.
    \end{align*}
    where the supremum is taken over $\{f : \|f\|_{L_t^1 L_x^2} = 1\}$. Interpolating with Lemma \ref{prel/strichartz-basic}, this completes the proof of Proposition \ref{prel/strichartz} in the case $q_2 \leq q_1$. Note that the Lorentz exponent $\frac{2q_1}{q_2}$ arises from the observation:
    \begin{equation*}
        L_t^{q_1,\frac{2q_1}{q_2}}\dot{H}_x^{p_1,\frac{\theta_1\alpha}{2}} = \Big( L_t^\infty L_x^2,\;L_t^{q_2,2}\dot{H}_x^{p_2,\frac{\theta_2\alpha}{2}}\Big)_{\left[\frac{q_2}{q_1}\right]},
    \end{equation*}
    in the sense of complex interpolation.
\end{proof}

\section{Spacetime Bounds}
        For the proof of Theorem \ref{theorem}, it will be necessary to control solutions to \eqref{gKdV} in a variety of mixed Lorentz spacetime norms.
    Notably, in the case of $4 \leq k < 8$ in Theorem \ref{spacetime/theorem}, we require that initial data lie in $\dot{H}^\frac{1}{4} \cap \dot{H}^{\frac{1}{4} - \frac{8-k}{4(k-1)}}$, which is not scaling critical. In particular, in the case $k = 4$, we require $u_0$ to lie in a negative regularity space, which introduces its own challenges.

    In this section, we prove the spacetime bounds we require. In doing so, we develop a persistence of negative (and positive) regularity for solutions to \eqref{gKdV}. As the negative regularity spaces introduce additional challenges --- such as the absence of a Leibniz rule --- and are absent from the standard well-posedness theory, we present the full details here.
    \begin{proposition}[Spacetime bounds]\label{spacetime/theorem}
        Let $k \geq 4$, $\theta \in (0,1)$, and $(q,p) = (\frac{4}{\theta},\frac{2}{1-\theta})$.
        Given $s \geq - \frac{5}{6}$ and $u_0 \in \dot{H}^{s_k} \cap \dot{H}^s(\R)$, let $u(t)$ be the corresponding solution to \eqref{gKdV}.
        Then
        \begin{equation}\label{spacetime/theorem 1}
            \big\||\partial_x|^{\frac{\theta}{4} + s} u\big\|_{L_t^{q,2 \vee \frac{q}{3}} L_x^{p}} \leq C\big(\|u_0\|_{\dot{H}^{s_k}}\big) \|u_0\|_{\dot{H}^s}.
        \end{equation}

        Moreover, for all $s \geq -\frac{5}{6}$ and $\theta \in [0,1]$ with $(p,q) = (\frac{4}{\theta}, \frac{2}{1-\theta})$,
        \begin{equation}\label{spacetime/theorem 2}
            \big\||\partial_x|^{1 - \frac{5\theta}{4} + s}u\big\|_{L_x^p L_t^q} \leq C\big(\|u_0\|_{\dot{H}^{s_k}}\big) \|u_0\|_{\dot{H}^s}.
        \end{equation}
    \end{proposition}

\begin{proof}
    We first focus on \eqref{spacetime/theorem 2} for $p = 30$, $q = \frac{30}{13}$, and $\theta = \frac{2}{3}$.
    By the density of Schwartz functions\footnote{In the case $s < -\frac{1}{2}$, we must also restrict to Schwartz initial data which vanish at $\xi = 0$.} in $\dot{H}^{s_k} \cap \dot{H}^s$, it suffices to consider Schwartz initial data $u_0$.
    
    To close our bootstrap argument, we first show that this norm is finite over any compact time interval $|K| < \infty$. By the backwards mixed-norm Sobolev embedding (Lemma \ref{proof/sobolev}) and persistence of positive regularity \cite{KPV93}, we find that for all $|K| < \infty$,
    \begin{align}\label{proof/finite}
        \hspace{-6pt}\big\||\partial_x|^{s + \frac{5}{6}}u \big\|_{L_x^{30}L_t^{\frac{30}{13}}(K)} & \lesssim \big\||\partial_x|^{s + 1} u\big\|_{L_x^5L_t^{\frac{30}{13}}(K)} \lesssim \big\||\partial_x|^{s + 1} u\big\|_{L_x^5L_t^{10}(K)}|K|^{\frac{1}{3}} < \infty.
    \end{align}
    We note that this estimate relies on higher regularity norms of $u_0$. As our resulting spacetime bounds depend only on the $\dot{H}^{s_k} \cap \dot{H}^s$ norm of $u_0$, this is inconsequential.

    We now proceed via a bootstrap argument. Fix a small parameter $\eta > 0$ to be chosen later based on absolute constants.
    By \eqref{gKdV-Dodson} for $k = 4$ and \eqref{KPV-gKdV-Global-kg4-A3} for $k > 4$, we may decompose $\R$ into $J(\|u_0\|_{\dot{H}^{s_k}},\eta)$-many (potentially unbounded) intervals $I_j$ on which
    \begin{equation*}
        \|u\|_{L_x^\frac{5k}{4} L_t^\frac{5k}{2}(I_j)} \lesssim \||\partial_x|^{\frac{1}{10}-\frac{2}{5k}}|\partial_t|^{\frac{3}{10}-\frac{6}{5k}}u\big\|_{L^{p_k}_{x}L^{q_k}_{t}(I_j)} < \eta,
    \end{equation*}
    where $\frac{1}{p_k}=\frac{2}{5k}+\frac{1}{10}$, $\frac{1}{q_k}=\frac{3}{10}-\frac{4}{5k}$.

    Fix a spacetime slab $I_j\times \R$ and some compact time interval $K \subset \R$. The backwards Strichartz inequality (Lemma \ref{gKdVL(XT)Norm}) followed by the fractional Leibniz rule (Lemma \ref{proof/leibniz}) then imply that for any $t_j \in I_j \cap K$,
    \begin{equation}\begin{split}\label{proof/1}
        \big\||\partial_x|^{s + \frac{5}{6}} u\big\|_{L_x^{30}L_t^{\frac{30}{13}}(K \cap I_j)}
        & \lesssim \big\| |\partial_x|^{s} u(t_j)\big\|_{L^2} + \big\| |\partial_x|^{s + \frac{5}{6}} u^{k+1}\big\|_{L_{x,t}^{6/5}(K \cap I_j)} \\
        & \lesssim \big\| |\partial_x|^s u(t_j)\big\|_{L^2} + \eta^k\big\||\partial_x|^{s + \frac{5}{6}}u \big\|_{L_x^{30} L_t^{\frac{30}{13}}(K \cap I_j)}.
    \end{split}\end{equation}
    We note that the fractional Leibniz rule is only valid because $s + \frac{5}{6} \geq 0$.

    Noting that $|K|<\infty$ and choosing $\eta$ sufficiently small relative to the constants in \eqref{proof/1}, \eqref{proof/finite} and a standard bootstrap argument imply that
    \begin{equation*}
        \big\||\partial_x|^{s + \frac{5}{6}} u\big\|_{L_x^{30}L_t^{\frac{30}{13}}(K \cap I_j)} \lesssim \big\| |\partial_x|^s u(t_j)\big\|_{L^2},
    \end{equation*}
    uniformly for $|K| < \infty$. Supremizing over $|K| < \infty$, we then find that
    \begin{equation*}
        \big\||\partial_x|^{s + \frac{5}{6}} u\big\|_{L_x^{30}L_t^{\frac{30}{13}}(I_j)} \lesssim \big\| |\partial_x|^s u(t_j)\big\|_{L^2}
    \end{equation*}
    for any $t_j \in I_j$.

    We now extend this to all norms present in Proposition \ref{spacetime/theorem}. Define $\mathcal{S}$ such that
    \begin{equation*}
        \|\cdot\|_{\mathcal{S}} = \big\||\partial_x|^{\frac{\theta}{4}} \cdot\big\|_{L_t^{q,\phi} L_x^{p}} \quad\text{or}\quad \|\cdot\|_{\mathcal{S}} = \big\||\partial_x|^{1 - \frac{5\theta}{4}}\cdot\big\|_{L_x^p L_t^q}
    \end{equation*}
    for $\theta, p, q, \phi$ as in Proposition \ref{spacetime/theorem}.
    For all $\mathcal{S}$, the Strichartz estimates, Proposition \ref{prel/strichartz} and Lemma \ref{gKdVL(XT)Norm}, along with \eqref{KPV-gKdV-Global-kg4-A3} then imply that
    \begin{align}
        \big\||\partial_x|^{s} u\big\|_{\mathcal{S}(I_j)} & \lesssim \big\| |\partial_x|^s u(t_j)\big\|_{L^2} + \big\| |\partial_x|^{s + \frac{5}{6}} u^{k+1}\big\|_{L_{t,x}^{\frac{6}{5}}(I_j)} \label{proof/8}\\
        & \lesssim \big\||\partial_x|^s u(t_j)\big\|_{L^2} + \|u\|^k_{L_x^\frac{5k}{4} L_t^\frac{5k}{2}(I_j)}\big\| |\partial_x|^{s + \frac{5}{6}} u\big\|_{L_x^{30} L_t^{\frac{30}{13}}(I_j)} \nonumber \\
        & \leq C\big(\|u_0\|_{L^2}\big)\big\| |\partial_x|^s u(t_j)\big\|_{L^2}. \nonumber
    \end{align}

    Using the $\mathcal{S} = L_t^\infty \dot{H}^s_x$ estimate, an iterative argument with appropriately chosen $t_j$ then implies that
    \begin{equation*}
        \big\||\partial_x|^{s} u\big\|_{\mathcal{S}(I_j)} \leq C\big(\|u_0\|_{\dot{H}^{s_k}},j\big)\big\| |\partial_x|^s u_0\big\|_{L^2}.
    \end{equation*}
    Summing over $j = 1,\dots,J(\|u_0\|_{\dot{H}^{s_k}})$, this concludes the proof of Proposition \ref{spacetime/theorem}.
\end{proof}

\section{Proof of dispersive decay}
    In this section, we prove Theorem \ref{theorem} and establish dispersive decay for solutions to \eqref{gKdV} for $k \geq 4$. In the case $k \geq 8$, this result is scaling-critical and will follow a proof similar to that of \cite{FKVZ24,Kowal24}. In the case $4 \leq k < 8$, adaptations are necessary to account for the lack of scaling-criticality.
\begin{proof}
    The proof for $4 \leq k < 8$ and $k \geq 8$ will follow identical structures with the only difference coming from the spaces used. To unify the arguments, let $\mathcal{H} = \dot{H}^{\frac{1}{4}} \cap \dot{H}^{\frac{1}{4} - \frac{8-k}{4(k-1)}}$ for $4 \leq k < 8$ and let $\mathcal{H} = \dot{H}^{s_k}$ for $k \geq 8$.

    It suffices to work with $t > 0$ as $t < 0$ will follow from time-reversal symmetry. By the density of Schwartz functions in $L^1 \cap \mathcal{H}$, it suffices to consider Schwartz solutions of \eqref{gKdV}.

    For $0 < T \leq \infty$, we define the norm
    \begin{equation*}
        \|u\|_{X(T)} = \sup_{t \in [0,T)} |t|^{\frac{1}{3}}\|u(t)\|_{L_x^\infty}.
    \end{equation*}
    It then suffices to show
    \begin{equation}\label{proof/conclusion}
        \|u\|_{X(\infty)} \leq C(\|u_0\|_{\mathcal{H}})\|u_0\|_{L^{1}},
    \end{equation}
    for which we proceed with a bootstrap argument.
    
    Let $\eta > 0$ denote a small parameter to be chosen later. For ease of notation, let $\theta = 1 \wedge \frac{8}{k}$. Define $q_\theta = \frac{4(k-1)}{2-\theta}$, $ r_\theta = \frac{4(k-1)}{4 - \theta}$, and $p_\theta = \frac{2(k-1)}{1 + \theta}$. We note that $r_\theta \geq \frac{q_\theta}{3} = \frac{4(k-1)}{3(2-\theta)}$. By the nesting of Lorentz spaces and Proposition \ref{spacetime/theorem}, we may then estimate
    \begin{align*}
        \|u\|_{L_t^{q_\theta,r_\theta}L_x^{p_\theta}} 
        & \lesssim \Big\||\partial_x|^{\frac{k-4}{2(k-1)}}u \Big\|_{L_t^{\frac{4(k-1)}{2-\theta}, \frac{4(k-1)}{3(2-\theta)}} L_x^{\frac{2(k-1)}{k+\theta-3}}}\\
        & \leq C\big(\|u_0\|_{\dot{H}^{s_k}}\big)\Big\||\partial_x|^{s_k - \frac{8-k\theta}{4k(k-1)}} u_0\Big\|_{L^2} = C(\|u_0\|_{\mathcal{H}}).
    \end{align*}
     This implies that we may decompose $[0,\infty)$ into $J(\|u_0\|_{\mathcal{H}},\eta)$ many intervals $I_j = [T_{j-1},T_j)$ on which
    \begin{equation}\label{proof/smallness}
        \big\||\partial_x|^{\frac{1}{2}}u\big\|_{L_t^\frac{4}{\theta} L_x^\frac{2}{1-\theta}(I_j)}, \|u\|_{L_t^{q_\theta,r_\theta}L_x^{p_\theta}(I_j)} < \eta.
    \end{equation}

    We aim to show that for all $j = 1,\dots, J(\|u_0\|_{\mathcal{H}}, \eta)$,
    \begin{equation}\label{proof/bootstrap}
        \|u\|_{X(T_j)} \lesssim \|u_0\|_{L^1} + C(\|u_0\|_{\mathcal{H}})\|u\|_{X(T_{j-1})} + \eta^k \|u\|_{X(T_j)}.
    \end{equation}
    Choosing $\eta > 0$ sufficiently small based on the constants in \eqref{proof/bootstrap}, we could then iterate over all indices $j = 1,\dots, J(\|u_0\|_{\mathcal{H}})$ to yield \eqref{proof/conclusion} and conclude the proof of Theorem \ref{theorem}.

    We therefore focus on \eqref{proof/bootstrap}. Fix $t \in [0,T_j)$ and recall the Duhamel formula:
    \begin{equation*}
        u(t) = e^{-t\partial_x^3} u_0 \mp i \int_0^{t} e^{-(t-s)\partial_x^3}\partial_x u^{k+1}(s) ds.
    \end{equation*}
    By the dispersive estimate \eqref{prel/dispersive} with $\alpha = 0$, the contribution of the linear term to $\|u(t)\|_{X(T_j)}$ is immediately seen to be acceptable:
    \begin{equation*}
        \big\|e^{-t\partial_x^3} u_0\big\|_{X(T_j)} \lesssim \|u_0\|_{L^1}.
    \end{equation*}
    We thus focus on the nonlinear correction.

    By the adapted dispersive estimate \eqref{prel/dispersive-adapted}, we find that
    \begin{align}\label{proof/failure 0}
         \bigg\|\int_0^t e^{-(t-s)\partial^3_x}\partial_x u^{k+1}(s)ds\bigg\|_{L_x^\infty}
         \lesssim \int_0^t |t-s|^{-\frac{1}{2}}\big\||\partial_x|^{\frac{1}{2}}u^{k+1}(s)\big\|_{L_x^1}ds.
    \end{align}
    Applying the fractional Leibniz rule \eqref{prel/leibniz}, we may then estimate
    \begin{equation}\label{proof/failure 1}
         \bigg\|\int_0^t e^{-(t-s)\partial^3_x}u^{k+1}(s)ds\bigg\|_{L_x^\infty}
         \lesssim \int_0^t |t-s|^{-\frac{1}{2}}\|u(s)\|_{L_x^\infty} \|u(s)\|_{L_x^{p_\theta}}^{k-1} \big\||\partial_x|^{\frac{1}{2}}u(s)\big\|_{L_x^\frac{2}{1-\theta}}ds.
    \end{equation}
    We note that in the case of $4 \leq k \leq 8$ (equiv.~$\theta = 1$), this requires the endpoint fractional Leibniz rule, Proposition \ref{prel/leibniz}, to be used twice as we begin with $L_x^1$ and estimate $|\partial_x|^{\frac{1}{2}}u$ in $L_x^\infty$; see, e.g., \eqref{prel/endpoint product}.

    By definition, $|s|^{\frac{1}{3}}\|u(s)\|_{L^\infty} \leq \|u\|_{X(s)}$. Then
    \begin{equation}\label{proof/failure 2}
        \begin{split}
        \bigg\|\int_0^t e^{-(t-s)\partial^3_x}& u^{k+1}(s)ds\bigg\|_{L_x^\infty} \\
         & \lesssim \int_0^t |t-s|^{-\frac{1}{2}}|s|^{-\frac{1}{3}}\|u\|_{X(s)} \|u(s)\|_{L_x^{p_\theta}}^{k-1} \big\||\partial_x|^{\frac{1}{2}}u(s)\big\|_{L_x^\frac{2}{1-\theta}}ds.
         \end{split}
    \end{equation}

    We now decompose $ [0,t) $ into $[0,\frac{t}{2})$ and $[\frac{t}{2},t)$. For $s \in  [0,\frac{t}{2}) $, we note that $|t-s| \sim |t|$, and for $s \in [\frac{t}{2},t)$, we note that $|s| \sim |t|$. We may then estimate
    \begin{equation*}\begin{split}
        \bigg\|\int_0^t e^{-(t-s)\partial^3_x}& u^{k+1}(s)ds\bigg\|_{L_x^\infty}\\
         & \lesssim |t|^{-\frac{1}{3}}\int_0^{\frac{t}{2}} |t-s|^{-\frac{1}{6}}|s|^{-\frac{1}{3}}\|u\|_{X(s)} \|u(s)\|_{L_x^{p_\theta}}^{k-1} \big\||\partial_x|^{\frac{1}{2}}u(s)\big\|_{L_x^\frac{2}{1-\theta}}ds \\
         & \hspace{11pt} + |t|^{-\frac{1}{3}}\int_{\frac{t}{2}}^t |t-s|^{-\frac{1}{2}}\|u\|_{X(s)} \|u(s)\|_{L_x^{p_\theta}}^{k-1} \big\||\partial_x|^{\frac{1}{2}}u(s)\big\|_{L_x^\frac{2}{1-\theta}}ds.
    \end{split}\end{equation*}
    By H\"older's inequality, we may place $|t-s|^{-\frac{1}{2}}$ into $L_s^{2,\infty}$, $|s|^{-\frac{1}{3}}$ into $L_s^{3,\infty}$ and $|t-s|^{-\frac{1}{6}}$ into $L_s^{6,\infty}$. Noting that $[0,t) \subset [0,T_j)$, this then implies that
    \begin{align*}
        \bigg\|\int_0^t e^{-(t-s)\partial^3_x}u^{k+1}(s)ds\bigg\|_{L_x^\infty}\hspace{-4pt}
         & \lesssim |t|^{-\frac{1}{3}}\Big\|\|u\|_{X(s)} \|u(s)\|_{L_x^{p_\theta}}^{k-1} \big\||\partial_x|^{\frac{1}{2}}u(s)\big\|_{L_x^\frac{2}{1-\theta}}\Big\|_{L_s^{2,1}([0,T_j))}.
    \end{align*}
    Here we note the necessity of the Lorentz spacetime bounds to control the singularity of the time decay at $s = 0$ and $t = s$.

    We now decompose $[0,T_j)$ into $[0,T_{j-1}) \cup I_j$ to find
    \begin{equation}\begin{split}\label{proof/2}
        \bigg\|\int_0^t e^{-(t-s)\partial^3_x}& u^{k+1}(s)ds \bigg\|_{L_x^\infty} \\
         & \lesssim |t|^{-\frac{1}{3}}\|u\|_{X(T_{j-1})}\|u\|^{k-1}_{L_t^{q_\theta,r_\theta}L_x^{p_\theta}}\big\||\partial_x|^{\frac{1}{2}}u\big\|_{L_t^\frac{4}{\theta} L_x^\frac{2}{1-\theta}} \\
         & \hspace{11pt} + |t|^{-\frac{1}{3}}\|u\|_{X(T_{j})}\|u\|^{k-1}_{L_t^{q_\theta,r_\theta}L_x^{p_\theta}(I_j)}\big\||\partial_x|^{\frac{1}{2}}u\big\|_{L_t^\frac{4}{\theta} L_x^\frac{2}{1-\theta}(I_j)}.
    \end{split}\end{equation}
    Supremizing over $t \in [0,T_j)$, Proposition \ref{spacetime/theorem} and \eqref{proof/smallness} then imply that
    \begin{align*}
        \bigg\|\int_0^t e^{-(t-s)\partial^3_x}u^{k+1}(s)ds\bigg\|_{X(T_j)}
         & \leq C\big(\|u_0\|_{\mathcal{H}}\big)\|u\|_{X(T_{j-1})} + \eta^k\|u\|_{X(T_{j})}.
    \end{align*}
    Combined with the linear estimate, this yields the bootstrap statement \eqref{proof/bootstrap} and concludes the proof of Theorem \ref{theorem}.
\end{proof}

\bibliographystyle{abbrv}
\bibliography{references}
\end{document}